\documentclass{amsart}
\usepackage{amsfonts}
\usepackage{amssymb}
\usepackage{amsmath}

\newtheorem{theorem}{Theorem}[section]
\newtheorem*{theorem A}{Theorem A}
\newtheorem*{theorem B}{N\"olker's Theorem}
\newtheorem{lemma}{Lemma}[section]
\newtheorem{proposition}{Proposition}[section]

\theoremstyle{remark}
\newtheorem{remark}{Remark}[section]
\theoremstyle{remark}

\theoremstyle{definition}
\newtheorem{definition}{Definition}[section]
\newtheorem{example}{Example}[section]
\numberwithin{equation}{section}
\def \({\left ( }
\def \){\right )}
\def \<{\left < }
\def \>{\right >}

\begin{document}
\title{Lightlike Hypersurfaces of Metallic Semi-Riemannian Manifolds}
\author{B\.{I}lal Eftal ACET}
\address{Faculty of Arts and Sciences, Department of Mathematics, Ad\i yaman
University,\ Ad\i yaman, TURKEY}
\email{eacet@adiyaman.edu.tr}
\date{}

\begin{abstract}

In our article, we introduce and study lightlike hypersurfaces 
of a metallic semi-Riemannian manifold. We
examine some geometric properties of invariant lightlike hypersurfaces. We show that the induced structure on an invariant lightlike
hypersurface is also metallic. We also define screen semi-invariant lightlike hypersurfaces, 
investigate integrability conditions for the distributions and give some examples.
\end{abstract}

\maketitle

\section{Introduction}

In terms of differential geometry, submanifold theory has an attraction for
geometers. One of the most important topic is the theory of lightlike
submanifolds. A submanifold of a semi-Riemann manifold is called a lightlike
submanifold if the induced metric is degenerate. So, geometry of lightlike
submanifold is very different from the non-degenerate submanifold. The
general view of lightlike submanifold has been introduced in \cite{DB}.
Later, K. L. Duggal and B. \c{S}ahin have been developed many new classes of
lightlike submanifolds on indefinite Kaehler \cite{DS3}, indefinite Sasakian 
\cite{DS1} manifolds and different applications of lightlike submanifolds 
\cite{DS}. On this subject, some applications of the theory of mathematical
physics, especially electromagnetisms \cite{DB}, black hole theory \cite{DS}
and general relativity \cite{G}, is inspired. Many studies on \ lightlike
submanifolds have been reported by many geometers (see \cite{APK3, APK1,
SBE,FBR-1, FBR, BC, FC} and the references therein).

Investigating submanifold theory on manifolds endowed with various geometric
structures provides a fruitful study field. Recently, Riemannian manifolds
with metallic structures are widely studied and metallic structures on
Riemannian manifolds provides many geometric results to characterize a
submanifold of such ambient manifolds.

As a generalization of the golden mean, in 2002, V. W. de Spinadel \cite%
{V.W4} introduced metallic means family which contains the silver mean, the
bronze mean, the copper mean and the nickel mean etc. For some positive
integer $p$ and $q,$ the positive solution of 
\begin{equation*}
x^{2}-px-q=0,
\end{equation*}%
is called a $(p,q)$-metallic number \cite{V.W1} which has the expression%
\begin{equation*}
\sigma _{p,q}=\frac{p+\sqrt{p^{2}+4q}}{2}.
\end{equation*}%
It is well-known that we have the golden mean $\phi =\frac{1+\sqrt{5}}{2}$,
for $p=q=1$ and the silver mean $\sigma _{2,1}=1+\sqrt{2}$, for $p=2,$ $q=1.$
The metallic mean family plays an important role to establish a relationship
between mathematics and architecture. For example silver and golden mean can
be seen in the sacred art of India, Egypt, China, Turkey and different
ancient civilizations \cite{HR-31}.

Polynomial structures on manifolds were introduced by S. I. Goldberg, K.
Yano and N. C. Petridis in \cite{G-Yano, G-Petri}. As a particular case of
polynomial structures which is called golden structure was defined in \cite%
{CR-HR-Chaos-2008, CR-HR-MUN} and some generalizations of this called
metallic structure in \cite{HR-CR-Revista}. Being inspired by the metallic
mean, the term of metallic manifold was studied in \cite{HR-CR-Revista} by a 
$(1,1)$-tensor field $\tilde{J}$ on $\tilde{N}$ which satisfies $\tilde{J}%
^{2}=p\tilde{J}+qI$, where $p$, $q$ are fixed positive integer numbers and $%
I $ is the identity operator on $\tilde{N}$. Moreover, if $(\tilde{N},\tilde{%
g})$ is a Riemannian manifold endowed with a metallic structure $\tilde{J}$
such that the Riemannian metric $\tilde{g}$ is $\tilde{J}$-compatible, i.e., 
$\tilde{g}(\tilde{J}X,Y)=\tilde{g}(X,\tilde{J}Y),$ for any $X,Y\in X(\tilde{N%
})$, then $(\tilde{g},\tilde{J})$ is called metallic Riemannian structure
and $(\tilde{N},\tilde{g},\tilde{J})$ is a metallic Riemannian manifold.
Metallic structure on the ambient Riemannian manifold provides important
geometrical results on the submanifolds, since it is an important tool while
investigating the geometry of submanifolds. Invariant, anti-invariant,
semi-invariant, slant and semi-slant submanifolds of a metallic Riemannian
manifold are studied in \cite{Adara-1, Adara-2, Adara-3} and the authors
obtained important characterizations on submanifolds of metallic Riemannian
manifolds.

One of the most important subclasses of metallic Riemannian manifolds is the
golden Riemannian manifolds. Many authors have studied golden Riemannian
manifolds and their submanifolds in recent years (see \cite{CR-HR-MUN,
Gezer-Cengiz, FC3, E1}. N. Poyraz \"{O}nen and E. Ya\c{s}ar \cite{NE}
initiated the study of lightlike geometry in Golden semi-Riemannian
manifolds, by investigating lightlike hypersurfaces of a golden
semi-Riemannian manifold.

Motivated by the studies on submanifolds of metallic Riemannian manifolds
and lightlike submanifolds of semi-Riemannian manifolds, in this paper we
introduce lightlike hypersurfaces of a metallic semi-Riemannian manifold.
Section 1 is devoted to preliminaries containing basic definitions for
metallic semi-Riemannian manifolds and lightlike hypersurfaces. Section 3 is
divided two subsections. Firstly we define invariant lightlike hypersurfaces
of a metallic semi-Riemannian manifold and prove that induced structures on
the invariant hypersurfaces are also metallic. In Subsection 2 we examine
screen semi-invariant lightlike hypersurfaces, give examples and investigate
integrability conditions for the distributions defined in such hypersurfaces.

\section{Preliminaries}

The positive solution of%
\begin{equation}
x^{2}-px-q=0,  \label{1}
\end{equation}%
is called member of the metallic means family \cite{V.W4} , where $p$, $q$
are fixed positive integers. These numbers denoted by;%
\begin{equation}
\sigma _{p,q}=\frac{p+\sqrt{p^{2}+4q}}{2},  \label{2}
\end{equation}%
are known $(p,q)$-metallic numbers.

A polynomial structure on a semi-Riemannian manifold $\tilde{N}$ is known
metallic if it is determined by $\tilde{J}$ such that 
\begin{equation}
\tilde{J}^{2}=p\tilde{J}+qI.  \label{3}
\end{equation}%
If a semi-Riemannian metric $\tilde{g}$ satisfies the equation%
\begin{equation}
\tilde{g}(U,\tilde{J}V)=\tilde{g}(\tilde{J}U,V),  \label{4}
\end{equation}%
which yields%
\begin{equation}
\tilde{g}(\tilde{J}U,\tilde{J}V)=p\tilde{g}(U,\tilde{J}V)+q\tilde{g}(U,V),
\label{5}
\end{equation}%
then $\tilde{g}$ is called $\tilde{J}$-compatible.

\begin{definition}
A semi-Riemannian manifold $(\tilde{N},\tilde{g})$ equipped with $\tilde{J}$
such that the semi-Riemannian metric $\tilde{g}$ is $\tilde{J}$-compatible
is named a metallic semi-Riemannian manifold and $(\tilde{g},\tilde{J})$ is
called a metallic semi-Riemannian structure on $\tilde{N}$.
\end{definition}

\begin{example}
\label{ex-1} Let $\tilde{N}=%
\mathbb{R}
_{3}^{7}$ be a semi-Euclidean space with coordinate system $%
(x_{1},x_{2},...,x_{7})$ and signature $(-,+,-,+,-,+,+)$.

Taking%
\begin{equation*}
\tilde{J}(x_{1},x_{2},...,x_{7})=((p-\sigma )x_{1},\sigma x_{2},(p-\sigma
)x_{3},\sigma x_{4},(p-\sigma )x_{5},\sigma x_{6},\sigma x_{7}),
\end{equation*}%
then we can see that 
\begin{equation*}
\tilde{J}^{2}=p\tilde{J}+qI.
\end{equation*}%
Therefore $\tilde{J}$ is a metallic structure on $\tilde{N}$.
\end{example}

\begin{remark}
Putting $p=1=q$ in (\ref{5}) then $(\tilde{g},\tilde{J})$ is a golden
(semi)-Riemannian structure (see \cite{HC}, \cite{E1}).
\end{remark}

A polynomial structure on $\tilde{N}$ defined by a smooth tensor field of
type $(1,1)$ induces a generalized almost product structure $\tilde{F}$, $%
\tilde{F}^{2}=I,$ on $\tilde{N}$.

\begin{proposition}
\cite{HR-CR-Revista} Every almost product structure $\tilde{F}$ induces two
metallic structure on $\tilde{N}$ given as follow:%
\begin{equation}
\tilde{J}_{1}=\frac{p}{2}I+\left( \frac{2\sigma _{p,q}-p}{2}\right) \tilde{F}%
,\text{ \ \ \ \ }\tilde{J}_{2}=\frac{p}{2}I-\left( \frac{2\sigma _{p,q}-p}{2}%
\right) \tilde{F}.  \label{6}
\end{equation}%
Conversely, every metallic structure $\tilde{J}$ on $\tilde{N}$ induces two
almost product structures:%
\begin{equation}
\tilde{F}=\pm \left( \frac{2}{2\sigma _{p,q}-p}\tilde{J}-\frac{p}{2\sigma
_{p,q}-p}I\right) .  \label{7}
\end{equation}
\end{proposition}

Let $\tilde{N}$ be a semi-Riemannian manifold with index $q$, $0<q<n+1,$ and 
$\acute{N}$ be a hypersurface of $\tilde{N}$, with $g=\tilde{g}\mid _{\acute{%
N}}$. Then $\acute{N}$ is a lightlike hypersurface of $\tilde{N},$ if the
metric $g$ is of rank $n$ and the orthogonal complement $T\acute{N}^{\bot }$
of $T\acute{N}$, given as 
\begin{equation*}
T\acute{N}^{\bot }=\bigcup_{p\in \acute{N}}\{V_{p}\in T_{p}\tilde{N}%
:g_{p}(U_{p},V_{p})=0,\forall U\in \Gamma (T_{p}\acute{N})\},
\end{equation*}%
is a distribution of rank $1$ on $\acute{N}$ \cite{DB}. $T\acute{N}^{\bot
}\subset T\acute{N}$ and then it coincides with the radical distribution $%
RadT\acute{N}=T\acute{N}\cap T\acute{N}^{\bot }$.

A complementary bundle of $T\acute{N}^{\bot }$ in $T\acute{N}$ is a
non-degenerate distribution of constant rank $n-1$ over $\acute{N},$ which
is known a\textit{\ screen distribution} and demonstrated with $S(T\acute{N}%
) $.

\begin{theorem}
\label{theo-1}\cite{DB} Let $(\acute{N},g,S(T\acute{N}))$ be a lightlike
hypersurface of a semi-Riemannian manifold $\tilde{N}$. Then there exists a
unique rank $1$ vector sub-bundle $ltr(T\acute{N})$ of $T\tilde{N}$, with
base space $N$, such that for every non-zero section $E$ of $RadT\acute{N}$
on a coordinate neighbourhood $\wp \subset \acute{N}$, there exists a
section $N$ of $ltr(T\acute{N})$ on $\wp $ satisfying:%
\begin{equation*}
\tilde{g}(N,N)=0,\text{ \ \ \ \ }\tilde{g}(N,W)=0,\text{ \ \ \ \ }\tilde{g}%
(N,E)=1,\text{ \ for }W\in \Gamma (S(T\acute{N})\mid _{\wp }.
\end{equation*}%
$ltr(T\acute{N})$ is called \textit{the lightlike transversal vector bundle}
of $\acute{N}$.
\end{theorem}

By the previous theorem, we can state:%
\begin{equation}
T\acute{N}=S(T\acute{N})\bot RadT\acute{N},  \label{8}
\end{equation}%
\begin{eqnarray}
T\tilde{N} &=&T\acute{N}\oplus ltr(T\acute{N})  \notag \\
&=&S(T\acute{N})\bot \{RadT\acute{N}\oplus ltr(T\acute{N})\}.  \label{9}
\end{eqnarray}%
Let $\omega :\Gamma (T\acute{N})\rightarrow \Gamma (S(T\acute{N}))$ be the
projection morphism. For $U,V\,\in \Gamma (T\acute{N}),$ we have 
\begin{equation}
\tilde{\nabla}_{U}V=\nabla _{U}V+B(U,V)N,  \label{10}
\end{equation}%
\begin{equation}
\tilde{\nabla}_{V}N=-A_{N}U+\tau (U)N,  \label{11}
\end{equation}%
\begin{equation}
\nabla _{U}\omega V=\nabla _{U}^{\ast }\omega V+C(U,\omega V)E,  \label{12}
\end{equation}%
\begin{equation}
\nabla _{U}E=-A_{E}^{\ast }U-\tau (V)E.  \label{13}
\end{equation}

For the induced connection $\nabla ,$ we have 
\begin{equation}
(\nabla _{U}g)(V,Z)=B(U,Z)\theta (V)+B(U,V)\theta (Z),  \label{15}
\end{equation}%
where $\theta $ is a differential $1$-form and 
\begin{equation}
\theta (U)=\tilde{g}(N,U).  \label{16}
\end{equation}

Also note that 
\begin{equation}
B(U,E)=0,  \label{14}
\end{equation}%
\begin{equation}
g(A_{E}^{\ast }U,PV)=B(U,PV),\ \ \ g\ (A_{E}^{\ast }U,N)=0,  \label{18*}
\end{equation}%
\begin{equation}
g(A_{N}U,PV)=C(U,PV)\text{, \ \ \ }g(A_{N}U,N)=0,\text{\ }  \label{19}
\end{equation}%
\begin{equation}
A_{E}^{\ast }E=0.  \label{17}
\end{equation}

\section{LIGHTLIKE HYPERSURFACES OF\ METALLIC SEMI-RIEMANNIAN MANIFOLDS}

Let $\acute{N}$ be a lightlike hypersurface of a metallic semi-Riemannian
manifold $(\tilde{N},\tilde{g},\tilde{J})$. For every $U\in \Gamma (T\acute{N%
})$ and $N\in \Gamma (ltr(T\acute{N}))$, we get%
\begin{equation}
\tilde{J}U=\varphi U+u(U)N,  \label{27}
\end{equation}%
\begin{equation}
\tilde{J}N=\xi +v(E)N,  \label{28}
\end{equation}%
where $\varphi U$, $\xi \in \Gamma (T\acute{N}),$ and $u$, $v$ are 1-forms
given by%
\begin{equation}
u(U)=g(U,\tilde{J}E),\text{ \ \ \ \ \ \ \ }v(U)=g(U,\tilde{J}N).  \label{29}
\end{equation}

\begin{lemma}
Let $\acute{N}$ be a lightlike hypersurface of $(\tilde{N},\tilde{g},\tilde{J%
})$. Then we have 
\begin{equation}
\varphi ^{2}U=p\varphi U+q(U)-u(U)\xi ,  \label{30}
\end{equation}%
\begin{equation}
u(\varphi U)=pu(U)-u(U)v(E),  \label{31}
\end{equation}%
\begin{equation}
\varphi \xi =p\xi -v(E)\xi ,  \label{32}
\end{equation}%
\begin{equation}
v(E)^{2}=pv(E)+q-u(\xi ),  \label{33}
\end{equation}%
\begin{equation}
g(\varphi U,V)=g(U,\varphi V)+u(V)\theta (U)-u(U)\theta (V),  \label{34}
\end{equation}%
\begin{eqnarray}
g(\varphi U,\varphi V) &=&pg(U,\varphi V)+qg(U,V)+pu(V)\theta (U)  \label{35}
\\
&&-u(V)g(\varphi U,N)-u(U)g(\varphi V,N).  \notag
\end{eqnarray}
\end{lemma}

\begin{definition}
A metallic semi-Riemannain structure $\tilde{J}$ is called a locally
metallic structure if $\tilde{J}$ is parallel, i.e., $\tilde{\nabla}\tilde{J}%
=0.$
\end{definition}

\begin{lemma}
Let $\acute{N}$ be a lightlike hypersurface of a locally metallic
semi-Riemannian manifold $(\tilde{N},\tilde{g},\tilde{J})$. Then we have 
\begin{equation}
(\nabla _{U}\varphi )V=u(V)A_{N}U+B(U,V)\xi ,  \label{36}
\end{equation}%
\begin{equation}
(\nabla _{U}u)V=B(U,V)v(E)-B(U,\varphi V)-\tau (U)u(V),  \label{37}
\end{equation}%
\begin{equation}
\nabla _{U}\xi =-\varphi A_{N}U+\tau (U)\xi +A_{N}Uv(E),  \label{38}
\end{equation}%
\begin{equation}
U(v(E))=-B(U,\xi )-u(A_{N}U).  \label{39}
\end{equation}
\end{lemma}

\subsection{Invariant Lightlike Hypersurfaces}

\begin{definition}
Let $\acute{N}$ be a lightlike hypersurface of a locally metallic
semi-Riemannian manifold $(\tilde{N},\tilde{g},\tilde{J})$. In that case $%
\acute{N}$ is called invariant hypersurface of $\tilde{N}$ if%
\begin{equation}
\begin{array}{c}
\tilde{J}(RadT\acute{N})=RadT\acute{N}, \\ 
\tilde{J}(ltr(T\acute{N}))=ltr(T\acute{N}).%
\end{array}
\label{40}
\end{equation}
\end{definition}

\begin{example}
\label{ex-2}Let $\tilde{N}=%
\mathbb{R}
_{2}^{5}$ be a semi-Euclidean space with coordinate system $%
(x_{1},x_{2},x_{3},x_{4},x_{5})$ and signature $(-,+,-,+,+)$. Taking%
\begin{equation*}
\tilde{J}(x_{1},x_{2},x_{3},x_{4},x_{5})=(\sigma x_{1},\sigma x_{2},\sigma
x_{3},\sigma x_{4},\sigma x_{5}),
\end{equation*}%
then we can see that 
\begin{equation*}
\tilde{J}^{2}=p\tilde{J}+qI,
\end{equation*}%
which implies $\tilde{J}$ is a metallic structure on $\tilde{N}.$

Now, consider a hypersurface $\acute{N}$ of $\tilde{N}$ with 
\begin{equation*}
x_{1}=\sigma x_{5}.
\end{equation*}%
Then $T\acute{N}$ of $\acute{N}$ is spanned by%
\begin{equation*}
\Phi _{1}=\frac{\partial }{\partial x_{2}},\text{ \ \ \ \ }\Phi _{2}=\frac{%
\partial }{\partial x_{3}},
\end{equation*}%
\begin{equation*}
\Phi _{3}=\frac{\partial }{\partial x_{4}},\text{ \ \ \ \ \ \ }\Phi
_{4}=\sigma \frac{\partial }{\partial x_{1}}+\frac{\partial }{\partial x_{5}}%
,
\end{equation*}%
So, $RadT\acute{N}$ and $ltr(T\acute{N})$ are given by%
\begin{equation*}
RadT\acute{N}=Sp\{E=\sigma \frac{\partial }{\partial x_{2}}+\sigma \frac{%
\partial }{\partial x_{3}}\},
\end{equation*}%
and%
\begin{equation*}
ltr(T\acute{N})=Sp\{N=\frac{1}{2\sigma ^{2}}(\sigma \frac{\partial }{%
\partial x_{2}}-\sigma \frac{\partial }{\partial x_{3}})\}.
\end{equation*}%
Thus we arrive at%
\begin{equation*}
\tilde{J}E=\sigma E\text{ \ \ \ \ \ \ and \ \ \ \ \ \ \ }\tilde{J}\acute{N}%
=\sigma N,
\end{equation*}%
which shows that $\acute{N}$ is an invariant lightlike hypersurface$.$
\end{example}

\begin{theorem}
Let $\acute{N}$ be a lightlike hypersurface of a locally metallic
semi-Riemannian manifold $(\tilde{N},\tilde{g},\tilde{J})$. Then the
structure $\varphi $ is a metallic structure on $\acute{N}$.
\end{theorem}

\begin{proof}
$\acute{N}$ is an invariant lightlike hypersurface iff%
\begin{equation*}
\tilde{J}U=\varphi U.
\end{equation*}%
So we get 
\begin{equation}
u(U)=0.  \label{41}
\end{equation}%
By use of (\ref{30}) and (\ref{34}), we arrive at%
\begin{equation*}
\varphi ^{2}U=p\varphi U+qU
\end{equation*}%
and 
\begin{equation*}
g(\varphi U,V)=g(U,\varphi V).
\end{equation*}
\end{proof}

\begin{theorem}
Let $\acute{N}$ be an invariant lightlike hypersurface of a locally metallic
semi-Riemannian manifold $(\tilde{N},\tilde{g},\tilde{J})$. Then 
\begin{equation}
i)\text{ }B(U,\tilde{J}V)=B(\tilde{J}U,V)=\tilde{J}B(U,V),  \label{4.12}
\end{equation}%
\begin{equation}
ii)\text{ }B(\tilde{J}U,\tilde{J}V)=pB(U,\tilde{J}V)+qB(U,V).  \label{4.13}
\end{equation}
\end{theorem}

\begin{proof}
$i)$ Since $\acute{N}$ is an invariant lightlike hypersurface of a locally
metallic semi-Riemannian manifold $\tilde{N}$, then 
\begin{equation}
\tilde{\nabla}_{U}\tilde{J}V=\nabla _{U}\tilde{J}V+B(U,\tilde{J}V)N,
\label{4.14}
\end{equation}%
and 
\begin{equation}
\tilde{\nabla}_{U}\tilde{J}V=\tilde{J}(\nabla _{U}V)+B(U,V)\tilde{J}N.
\label{4.15}
\end{equation}%
From (\ref{4.14}) and (\ref{4.15}), we obtain $(i).$

$ii)$ It is obvious from (\ref{4.12}) with equation (\ref{5}).
\end{proof}

\subsection{Screen Semi-Invariant Hypersurfaces}

\begin{definition}
Let $\acute{N}$ be a lightlike hypersurface of a locally metallic
semi-Riemannian manifold $(\tilde{N},\tilde{g},\tilde{J})$. If%
\begin{equation}
\begin{array}{c}
\tilde{J}(RadT\acute{N})\subset S(T\acute{N}), \\ 
\tilde{J}(ltr(T\acute{N}))\subset S(T\acute{N}),%
\end{array}
\label{4.16}
\end{equation}%
then $\acute{N}$ is called a screen semi-invariant hypersurface of $\tilde{N}
$.
\end{definition}

\begin{example}
Let $\tilde{N}=%
\mathbb{R}
_{2}^{5}$ be a five dimensional metallic semi-Riemannian manifold with the
structure $(\tilde{N},\tilde{g},\tilde{J})$ given in Example \ref{ex-1}.
Consider a hypersurface $\acute{N}$ of $\tilde{N}$ with 
\begin{equation*}
x_{5}=\sigma x_{1}+\sigma x_{2}+x_{3}.
\end{equation*}%
Then $T\acute{N}$ of $\acute{N}$ is spanned by%
\begin{equation*}
\Phi _{1}=\frac{\partial }{\partial x_{1}}+\sigma \frac{\partial }{\partial
x_{5}},\text{ \ \ \ \ }\Phi _{2}=\frac{\partial }{\partial x_{2}}+\sigma 
\frac{\partial }{\partial x_{5}}
\end{equation*}%
\begin{equation*}
\Phi _{3}=\frac{\partial }{\partial x_{3}}+\frac{\partial }{\partial x_{5}},%
\text{ \ \ \ \ \ \ }\Phi _{4}=\frac{\partial }{\partial x_{4}}.
\end{equation*}%
The radical distribution $RadT\acute{N}$ and lightlike transversal
distribution $ltr(T\acute{N})$ are given by%
\begin{equation*}
RadT\acute{N}=Sp\{E=\sigma \Phi _{1}-\sigma \Phi _{2}+\Phi _{3}\},
\end{equation*}%
\begin{equation*}
ltr(T\acute{N})=Sp\left\{ N=\frac{1}{2}\left( -\sigma \frac{\partial }{%
\partial x_{1}}+\sigma \frac{\partial }{\partial x_{2}}-\frac{\partial }{%
\partial x_{3}}+\frac{\partial }{\partial x_{5}}\right) \right\} .
\end{equation*}%
It follows that $S(T\acute{N})$ is spanned by $\{\Omega _{1},\Omega
_{2},\Omega _{3}\}$, where%
\begin{equation*}
\Omega _{1}=-q\frac{\partial }{\partial x_{1}}+q\frac{\partial }{\partial
x_{2}}+\sigma \frac{\partial }{\partial x_{3}}+\sigma \frac{\partial }{%
\partial x_{5}},
\end{equation*}%
\begin{equation*}
\Omega _{2}=\frac{1}{2}\left\{ 
\begin{array}{c}
-\sigma (p-\sigma )\frac{\partial }{\partial x_{1}}+\sigma (p-\sigma )\frac{%
\partial }{\partial x_{2}} \\ 
-\sigma \frac{\partial }{\partial x_{3}}+\sigma \frac{\partial }{\partial
x_{5}}%
\end{array}%
\right\} ,
\end{equation*}%
\begin{equation*}
\Omega _{3}=\frac{\partial }{\partial x_{4}}.
\end{equation*}%
Thus we arrive at%
\begin{equation*}
\Omega _{1}=\tilde{J}E\text{ \ \ \ \ \ \ \ \ and \ \ \ \ \ \ \ }\Omega _{2}=%
\tilde{J}N,
\end{equation*}%
which imply that $\acute{N}$ is a screen semi-invariant hypersurface of $%
\tilde{N}.$
\end{example}

\begin{example}
Let $\tilde{N}=%
\mathbb{R}
_{2}^{5}$ be a $5$-dimensional metallic semi-Riemannian manifold with the
structure $(\tilde{N},\tilde{g},\tilde{J})$ given in Example \ref{ex-2}. If
we take a hypersurface $\acute{N}$ of $\tilde{N}$ given by 
\begin{equation*}
x_{5}=\sigma x_{3}+\sigma x_{4}+x_{1},
\end{equation*}%
then $T\acute{N}$ of $\acute{N}$ is spanned by%
\begin{equation*}
\Phi _{1}=\frac{\partial }{\partial x_{1}}+\frac{\partial }{\partial x_{5}},%
\text{ \ \ \ \ }\Phi _{2}=\frac{\partial }{\partial x_{2}},
\end{equation*}%
\begin{equation*}
\Phi _{3}=\frac{\partial }{\partial x_{3}}+\sigma \frac{\partial }{\partial
x_{5}},\text{ \ \ \ \ \ \ }\Phi _{4}=\frac{\partial }{\partial x_{4}}+\sigma 
\frac{\partial }{\partial x_{5}}.
\end{equation*}%
The radical distribution $RadT\acute{N}$ and lightlike transversal
distribution $ltr(T\acute{N})$ are given by%
\begin{equation*}
RadT\acute{N}=Sp\{E=\sigma \Phi _{3}-\sigma \Phi _{4}+\Phi _{1}\},
\end{equation*}%
\begin{equation*}
ltr(T\acute{N})=Sp\left\{ N=\frac{1}{2}(-\frac{\partial }{\partial x_{1}}-%
\frac{\partial }{\partial x_{3}}+\frac{\partial }{\partial x_{4}}+\frac{%
\partial }{\partial x_{5}})\right\} ,
\end{equation*}%
respectively. It follows that $S(T\acute{N})$ is spanned by $\{\Omega
_{1},\Omega _{2},\Omega _{3}\}$, where%
\begin{equation*}
\Omega _{1}=\frac{1}{2}\left\{ 
\begin{array}{c}
-\sigma \frac{\partial }{\partial x_{1}}-\sigma \frac{\partial }{\partial
x_{3}} \\ 
+\sigma \frac{\partial }{\partial x_{4}}+\sigma \frac{\partial }{\partial
x_{5}}%
\end{array}%
\right\} ,
\end{equation*}%
\begin{equation*}
\Omega _{2}=\sigma \frac{\partial }{\partial x_{1}}+\sigma ^{2}\frac{%
\partial }{\partial x_{2}}-\sigma ^{2}\frac{\partial }{\partial x_{4}}%
+\sigma \frac{\partial }{\partial x_{5}},
\end{equation*}%
\begin{equation*}
\Omega _{3}=\frac{\partial }{\partial x_{2}}.
\end{equation*}%
Thus we arrive at%
\begin{equation*}
\Omega _{1}=\tilde{J}N\text{ \ \ \ \ \ \ \ \ and \ \ \ \ \ \ \ }\Omega _{2}=%
\tilde{J}E,
\end{equation*}%
which imply that $\acute{N}$ is a screen semi-invariant hypersurface of $%
\tilde{N}.$
\end{example}

Since $S(T\acute{N})$ is non-degenerate, we can define an $(n-2)$%
-dimensional distribution $\mu _{0}$ such that%
\begin{equation}
S(T\acute{N})=\mu _{0}\bot \{\tilde{J}(RadT\acute{N})\oplus \tilde{J}(ltr(T%
\acute{N}))\},  \label{4.17}
\end{equation}%
from which 
\begin{equation}
T\acute{N}=\{\tilde{J}(RadT\acute{N})\oplus \tilde{J}(ltr(T\acute{N}))\}\bot
\mu _{0}\bot Rad(T\acute{N}),  \label{4.18}
\end{equation}%
\begin{equation}
T\tilde{N}=\{\tilde{J}(RadT\acute{N})\oplus \tilde{J}(ltr(T\acute{N}))\}\bot
\mu _{0}\bot \{Rad(T\acute{N})\oplus ltr(T\acute{N})\}.  \label{4.19}
\end{equation}%
Taking $\mathring{D}=Rad(T\acute{N})\bot \tilde{J}(Rad(T\acute{N}))\bot \mu
_{0}$ and $\mathring{D}^{\prime }=\tilde{J}(ltr(T\acute{N}))$ on $\acute{N},$
we get%
\begin{equation}
TN=\mathring{D}\oplus \mathring{D}^{\prime }.  \label{4.20}
\end{equation}%
Let $\zeta =\tilde{J}N$ and $\psi =\tilde{J}E$ be \ local lightlike vector
fields. For $U\in \Gamma (T\acute{N}),$ we can write%
\begin{equation}
U=QU+RU,  \label{4.21}
\end{equation}%
where $Q$ and $R$ are projections of $T\acute{N}$ into $\mathring{D}$ and $%
\mathring{D}^{\prime }$, respectively.

Moreover, for $U,V\in \Gamma (T\acute{N}),$ $\zeta \in \Gamma (\mathring{D}%
^{\prime })$ and $\psi \in \Gamma (\mathring{D}),$ we get%
\begin{equation}
\varphi ^{2}U=p\varphi U+q(U)-u(U)\zeta ,  \label{4.22}
\end{equation}%
\begin{equation}
u(\varphi U)=pu(U),\text{ \ \ \ \ }u(\zeta )=q,  \label{4.23}
\end{equation}%
\begin{equation}
g(\varphi U,V)=g(U,\varphi V)+u(V)\theta (U)-u(U)\theta (V),  \label{4.24}
\end{equation}%
\begin{eqnarray}
g(\varphi U,\varphi V) &=&pg(U,\varphi V)+qg(U,V)+pu(V)\theta (U)  \notag \\
&&-u(V)g(\varphi U,N)-u(U)g(\varphi V,N),  \label{4.25}
\end{eqnarray}%
\begin{equation}
(\nabla _{U}\varphi )V=u(V)A_{N}U+g(A_{E}^{\ast }U,V)\zeta ,  \label{4.26}
\end{equation}%
\begin{equation}
(\nabla _{U}u)V=-B(U,\varphi V)-u(V)\tau (U),  \label{4.27}
\end{equation}%
\begin{equation}
\nabla _{U}\zeta =-\varphi A_{N}U+\tau (U)\zeta ,  \label{4.28}
\end{equation}%
\begin{equation}
\nabla _{U}\psi =-\varphi A_{E}^{\ast }U-\tau (U)\psi ,  \label{4.29}
\end{equation}%
\begin{equation}
B(U,\zeta )=-C(U,\psi ).  \label{4.30}
\end{equation}

\begin{theorem}
Assume that $\acute{N}$ is a screen semi-invariant lightlike hypersurface of
a locally metallic semi-Riemannian manifold $(\tilde{N},\tilde{g},\tilde{J})$%
. Then lightlike vector field $\psi $ is parallel on $\acute{N}$ iff

$i)$ $\acute{N}$ is totally geodesic on $\tilde{N},$

$ii)$ $\tau =0.$
\end{theorem}

\begin{proof}
Let $\psi $ be a parallel vector field. In view of (\ref{27}) and (\ref{4.29}%
), for $U\in \Gamma (T\acute{N}),$ we get%
\begin{eqnarray}
0 &=&-\varphi A_{E}^{\ast }U-\tau (U)\psi  \notag \\
&=&-\tilde{J}A_{E}^{\ast }U-\tau (U)\psi +u(A_{E}^{\ast }U)N.  \label{4.31}
\end{eqnarray}%
Applying $\tilde{J}$ to (\ref{4.31}) and from (\ref{27}) with (\ref{3}), we
find 
\begin{eqnarray}
&&-p\varphi (A_{E}^{\ast }U)-pu(A_{E}^{\ast }U)N-qA_{E}^{\ast }U  \notag \\
&&-p\tau (U)\psi -q\tau (U)E-u(A_{E}^{\ast }U)\zeta  \label{4.32} \\
&=&0.  \notag
\end{eqnarray}%
Taking tangential and transversal part of equation (\ref{4.32}), we obtain%
\begin{equation*}
qA_{E}^{\ast }U=-q\tau (U)E-u(A_{E}^{\ast }U)\zeta ,\text{ \ \ \ \ }%
pu(A_{E}^{\ast }U)=0,
\end{equation*}%
which yields 
\begin{equation*}
A_{E}^{\ast }U=0,\text{ \ \ \ \ \ }\tau (U)=0.
\end{equation*}
\end{proof}

\begin{theorem}
Let $\acute{N}$ be a screen semi-invariant lightlike hypersurface of a
locally metallic semi-Riemannian manifold $(\tilde{N},\tilde{g},\tilde{J})$.
Then lightlike vector field $\zeta $ is parallel on $\acute{N}$ iff $\acute{N%
}$ and $S(T\acute{N})$ is totally geodesic on $\tilde{N}.$
\end{theorem}

\begin{proof}
Since $\zeta $ is a parallel vector field$,$ by use of (\ref{27}) and (\ref%
{4.28}), for $U\in \Gamma (T\acute{N})$ we get%
\begin{eqnarray}
0 &=&-\varphi A_{N}U-\tau (U)\zeta  \label{4.33} \\
&=&-\tilde{J}A_{N}U-\tau (U)\zeta +u(A_{N}U)N.  \notag
\end{eqnarray}%
Applying $\tilde{J}$ to (\ref{4.33}) and using (\ref{27}) with (\ref{3}), we
have 
\begin{eqnarray}
&&-p\varphi (A_{N}U)-pu(A_{N}U)N-qA_{N}U  \label{4.34} \\
&&-p\tau (U)\zeta -q\tau (U)N-u(A_{N}U)\zeta  \notag \\
&=&0.  \notag
\end{eqnarray}%
Therefore, from (\ref{4.34}), we obtain%
\begin{equation*}
qA_{N}U=u(A_{N}U)\zeta ,\text{ \ \ \ \ }pu(A_{N}U)=q\tau (U).
\end{equation*}%
So we get the proof of our assertion.
\end{proof}

\begin{definition}
Let $\acute{N}$ be a screen semi-invariant lightlike hypersurface of a
locally metallic semi-Riemannian manifold $(\tilde{N},\tilde{g},\tilde{J})$.
If $B(U,V)=0$, for any $U\in \Gamma (\mathring{D})$ and $V\in \Gamma (%
\mathring{D}^{\prime })$, then $\acute{N}$ is called a mixed geodesic
lightlike hypersurface.
\end{definition}

\begin{theorem}
Let $\acute{N}$ be a screen semi-invariant lightlike hypersurface of a
locally metallic semi-Riemannian manifold $(\tilde{N},\tilde{g},\tilde{J})$.
Then $\acute{N}$ is a mixed geodesic lightlike hypersurface if and only if

$i)$ There is no component of $A_{N},$ $\mathring{D}-$valuable.

$ii)$ There is no component of $A_{E}^{\ast },$ $\mathring{D}^{\prime }-$%
valuable.
\end{theorem}

\begin{proof}
Suppose that $\acute{N}$ is mixed geodesic, i.e.,%
\begin{equation}
B(U,\zeta )=0.  \label{4.35}
\end{equation}

By use of (\ref{4}) and (\ref{11}) in (\ref{4.35}), we get%
\begin{eqnarray*}
0 &=&B(U,\zeta )=B(U,\tilde{J}N) \\
&=&\tilde{g}(\tilde{\nabla}_{U}\tilde{J}N,E) \\
&=&\tilde{g}((\tilde{\nabla}_{U}\tilde{J})N+\tilde{J}\tilde{\nabla}_{U}N,E)
\\
&=&\tilde{g}(\tilde{\nabla}_{U}N,\tilde{J}E) \\
&=&-\tilde{g}(A_{N}U,\tilde{J}E),
\end{eqnarray*}%
from which we obtain $(i).$

Since 
\begin{equation*}
-\tilde{g}(A_{N}U,\tilde{J}E)=\tilde{g}(A_{E}^{\ast }U,\tilde{J}N),
\end{equation*}%
we arrive at $(ii)$.
\end{proof}

Now, we consider the distribution $\mu _{0}$, defined in (\ref{4.17}). In
view of (\ref{4.18}) and taking 
\begin{equation*}
\beta =\{\tilde{J}(Rad(T\acute{N}))\oplus \tilde{J}(ltr(T\acute{N}))\}\bot
Rad(T\acute{N}),
\end{equation*}%
for any $U\in \Gamma (T\acute{N})$, $V\in \Gamma (\mu _{0})$ and $Z\in
\Gamma (\beta )$, we can state%
\begin{equation}
\nabla _{U}V=\overset{\mu _{0}}{\nabla }_{U}V+\overset{\mu _{0}}{h}(U,V),
\label{4.36}
\end{equation}%
\begin{equation}
\nabla _{U}Z=-\overset{\mu _{0}}{A}_{Z}U+\nabla _{U}^{\bot }Z,  \label{4.37}
\end{equation}%
where $\overset{\mu _{0}}{\nabla }$ is a linear connection on $\mu _{0}$, $%
\overset{\mu _{0}}{h}:\Gamma (T\acute{N})\times \Gamma (\mu _{0})\rightarrow
\Gamma (\beta )$ is an $\Im (\acute{N})$ bilinear, $\overset{\mu _{0}}{A}$
is an $\Im (\acute{N})$ linear operator on $\Gamma (\mu _{0}),$
respectively, and $\nabla ^{\bot }$ is a linear connection on $\beta $.

Let $\wp \subset \acute{N}$ be a coordinate neighborhood. Then according to
decomposition given by (\ref{4.18}), we take%
\begin{equation}
\begin{array}{cc}
\alpha _{1}(U,V) & =g(\overset{\mu _{0}}{h}(U,V),\tilde{J}N), \\ 
\alpha _{2}(U,V) & =g(\overset{\mu _{0}}{h}(U,V),\tilde{J}E), \\ 
\alpha _{3}(U,V) & =g(\overset{\mu _{0}}{h}(U,V),N),%
\end{array}
\label{4.38}
\end{equation}%
for every $U,V\in \Gamma (\mu _{0}\,|_{\wp })$. Thus we can write equation (%
\ref{4.36}) by 
\begin{equation}
\nabla _{U}V=\overset{\mu _{0}}{\nabla }_{U}V+\frac{1}{q}\alpha _{1}(U,V)%
\tilde{J}E+\frac{1}{q}\alpha _{2}(U,V)\tilde{J}N+\alpha _{3}(U,V)E.
\label{4.39}
\end{equation}%
We shall compute $\alpha _{1},\alpha _{2}$ and $\alpha _{3}$ in terms of $B$
and $C$. Starting with%
\begin{eqnarray*}
g(\nabla _{U}V,\tilde{J}N) &=&g(\overset{\mu _{0}}{\nabla }_{U}V+\frac{1}{q}%
\alpha _{1}(U,V)\tilde{J}E+\frac{1}{q}\alpha _{2}(U,V)\tilde{J}N+\alpha
_{3}(U,V)E,\tilde{J}N) \\
&=&\alpha _{1}(U,V).
\end{eqnarray*}%
Then by use of (\ref{10}) and (\ref{11}), we get%
\begin{eqnarray*}
g(\nabla _{U}V,\tilde{J}N) &=&g(\tilde{J}\nabla _{U}V,N) \\
&=&g(\tilde{J}\tilde{\nabla}_{U}V,N) \\
&=&g((\tilde{\nabla}_{U}\tilde{J})V+\tilde{J}\tilde{\nabla}_{U}V,N) \\
&=&g(\tilde{J}\tilde{\nabla}_{U}V,N) \\
&=&-g(A_{N}U,\tilde{J}V)=-C(U,\tilde{J}V).
\end{eqnarray*}

Next we find%
\begin{eqnarray*}
g(\nabla _{U}V,\tilde{J}E) &=&g(\overset{\mu _{0}}{\nabla }_{U}V+\frac{1}{q}%
\alpha _{1}(U,V)\tilde{J}E+\frac{1}{q}\alpha _{2}(U,V)\tilde{J}N+\alpha
_{3}(U,V)E,\tilde{J}E) \\
&=&\alpha _{2}(U,V),
\end{eqnarray*}%
and from (\ref{10}) and (\ref{13}), we obtain

\begin{eqnarray*}
g(\nabla _{U}V,\tilde{J}E) &=&g(\tilde{J}\nabla _{U}V,E) \\
&=&g(\tilde{J}\tilde{\nabla}_{U}V,E) \\
&=&g((\tilde{\nabla}_{U}\tilde{J})V+\tilde{J}\tilde{\nabla}_{U}V,E) \\
&=&g(\tilde{J}V,\tilde{\nabla}_{U}E)=B(U,\tilde{J}V).
\end{eqnarray*}

By a similar way, we compute 
\begin{eqnarray*}
g(\nabla _{U}V,N) &=&g(\overset{\mu _{0}}{\nabla }_{U}V+\frac{1}{q}\alpha
_{1}(U,V)\tilde{J}E+\frac{1}{q}\alpha _{2}(U,V)\tilde{J}N+\alpha
_{3}(U,V)E,N) \\
&=&\alpha _{3}(U,V),
\end{eqnarray*}%
and%
\begin{equation*}
g(\nabla _{U}V,N)=-C(U,V).
\end{equation*}%
So, we can rewrite equation (\ref{4.39}) with 
\begin{equation}
\nabla _{U}V=\overset{\mu _{0}}{\nabla }_{U}V-C(U,\tilde{J}V)\tilde{J}E-B(U,%
\tilde{J}V)\tilde{J}N-C(U,V)E,  \label{4.40}
\end{equation}

\begin{theorem}
Let $\acute{N}$ be a screen semi-invariant lightlike hypersurface of a
locally metallic semi-Riemannian manifold $(\tilde{N},\tilde{g},\tilde{J})$.
Then $\mu _{0}$ is integrable if and only if%
\begin{equation}
C(\tilde{J}U,V)=C(U,\tilde{J}V),\text{ \ \ }B(\tilde{J}U,V)=B(U,\tilde{J}V),%
\text{ \ \ }C(U,V)=C(V,U),  \label{4.41}
\end{equation}%
for every $U,V\in \Gamma (\mu _{0})$.
\end{theorem}

\begin{proof}
Because of $\nabla $ is linear connection, by using (\ref{4.40}) we get%
\begin{eqnarray*}
\lbrack U,V] &=&\overset{\mu _{0}}{\nabla }_{U}V-\overset{\mu _{0}}{\nabla }%
_{V}U \\
&&+(C(U,\tilde{J}V)-C(\tilde{J}U,V))\tilde{J}E \\
&&+(B(U,\tilde{J}V)-B(\tilde{J}U,V))\tilde{J}N \\
&&+(C(U,V)-C(V,U))E.
\end{eqnarray*}%
If $\mu _{0}$ is integrable then the components of $[U,V]$ with respect to $%
\tilde{J}E,$ $\tilde{J}N$ and $E$ vanish. So, we get proof of our asssertion.

Conversely, if (\ref{4.41}) is satisfied we get%
\begin{equation*}
\lbrack U,V]\in \Gamma (\mu _{0}).
\end{equation*}
\end{proof}

\begin{theorem}
Let $\acute{N}$ be a screen semi-invariant lightlike hypersurface of a
locally metallic semi-Riemannian manifold $(\tilde{N},\tilde{g},\tilde{J})$.
Then the distribution $\mathring{D}$ is integrable if and only if%
\begin{equation}
B(\tilde{J}U,\tilde{J}V)=pB(V,\tilde{J}U)+qB(V,U),  \label{4.42}
\end{equation}%
for every $U,V\in \Gamma (\mathring{D})$.
\end{theorem}

\begin{proof}
Taking $U,V\in \Gamma (\mathring{D})$, \ we get $\tilde{J}U\in \Gamma (%
\mathring{D})$. Then $\mathring{D}$ is integrable iff 
\begin{eqnarray*}
\tilde{g}([\tilde{J}U,V],\tilde{J}E) &=&\tilde{g}(\tilde{\nabla}_{\tilde{J}%
U}V,\tilde{J}E)-\tilde{g}(\tilde{\nabla}_{V}\tilde{J}U,\tilde{J}E) \\
&=&\tilde{g}(\tilde{J}\tilde{\nabla}_{\tilde{J}U}V,E)-\tilde{g}(\tilde{J}%
\tilde{\nabla}_{V}U,\tilde{J}E) \\
&=&\tilde{g}(-((\tilde{\nabla}_{\tilde{J}U}\tilde{J})V+\tilde{\nabla}_{%
\tilde{J}U}\tilde{J}V,E) \\
&&-p\tilde{g}(\tilde{\nabla}_{V}U,\tilde{J}E)-q\tilde{g}(\tilde{\nabla}%
_{V}U,E) \\
&=&\tilde{g}(\tilde{\nabla}_{\tilde{J}U}\tilde{J}V,E)-p\tilde{g}(\tilde{%
\nabla}_{V}U,\tilde{J}E)-q\tilde{g}(\tilde{\nabla}_{V}U,E) \\
&=&B(\tilde{J}U,\tilde{J}V)-pB(V,\tilde{J}U)-qB(V,U),
\end{eqnarray*}%
which gives (\ref{4.42}).
\end{proof}

\begin{theorem}
Let $\acute{N}$ be a screen semi-invariant lightlike hypersurface of a
locally metallic semi-Riemannian manifold $(\tilde{N},\tilde{g},\tilde{J})$.
Then $\mathring{D}$ is parallel iff $\mathring{D}$ is totally geodesic on $%
\acute{N}$.
\end{theorem}

\begin{proof}
From the definition of the distribution $\mathring{D}$ we know that $%
\mathring{D}$ is parallel if and only if%
\begin{equation*}
g(\nabla _{U}V,\psi )=0.
\end{equation*}%
So, we have%
\begin{eqnarray*}
0 &=&g(\nabla _{U}V,\psi ) \\
&=&\tilde{g}(\nabla _{U}V,\psi ) \\
&=&\tilde{g}(\nabla _{U}V,\tilde{J}E) \\
&=&\tilde{g}(\tilde{J}\nabla _{U}V,E) \\
&=&\tilde{g}(-((\tilde{\nabla}_{U}\tilde{J})V+\tilde{\nabla}_{U}\tilde{J}V,E)
\\
&=&\tilde{g}(\tilde{\nabla}_{U}\tilde{J}V,E)=B(U,\tilde{J}V).
\end{eqnarray*}%
Thus we prove the assertion.
\end{proof}

\begin{definition}
Let $\acute{N}$ be a lightlike hypersurface of a locally metallic
semi-Riemannian manifold $(\tilde{N},\tilde{g},\tilde{J})$. If for every $%
U,V\in \Gamma (T\acute{N})$%
\begin{equation}
C(U,\omega V)=\phi B(U,V),  \label{4.43}
\end{equation}%
or $A_{N}=\phi A_{E}^{\ast }$, then $\acute{N}$ is called screen conformal
lightlike hypersurface, where $\phi $ is a non-vanishing smooth function 
\cite{DS}.
\end{definition}

From (\ref{4.30}) with (\ref{4.43}), we find%
\begin{equation}
B(U,\zeta +\phi \psi )=0.  \label{4.44}
\end{equation}

\begin{theorem}
Let $\acute{N}$ be a screen conformal screen semi-invariant ightlike
hypersurface of a locally metallic semi-Riemannian manifold $(\tilde{N},%
\tilde{g},\tilde{J})$. If $\acute{N}$ or screen distribution $S(T\acute{N})$
is totally umbilical then $\acute{N}$ is totally geodesic in $\tilde{N}.$
\end{theorem}

\begin{proof}
Let $\acute{N}$ be totally umbilical. Then we know that 
\begin{equation}
B(U,V)=\lambda g(U,V).  \label{4.45}
\end{equation}%
From (\ref{4.44}), we get\ 
\begin{equation}
\lambda g(U,\zeta +\phi \psi )=0.  \label{4.46}
\end{equation}%
Putting $U=\psi $ in (\ref{4.46}) we have $\lambda =0$ and $U=\zeta $ in (%
\ref{4.46}) we have $\lambda \phi =0.$ So we get $B=0=C.$

Now, if $S(T\acute{N})$ is totally umbilical we arrive at $B=0=C.$
\end{proof}

\end{document}